\documentclass[reqno]{amsart}
\makeatletter
\def\oversortoftilde#1{\mathop{\vbox{\m@th\ialign{##\crcr\noalign{\kern3\p@}%
      \sortoftildefill\crcr\noalign{\kern3\p@\nointerlineskip}%
      $\hfil\displaystyle{#1}\hfil$\crcr}}}\limits}

\def\sortoftildefill{$\m@th \setbox\z@\hbox{$\braceld$}%
  \braceld\leaders\vrule \@height\ht\z@ \@depth\z@\hfill\braceru$}

\makeatother

\usepackage{amscd, amsfonts, amsmath,amssymb,amsthm, mathtools} 
\usepackage{tabularx,ltablex}
\usepackage[all,color]{xy}
\usepackage[hidelinks]{hyperref}
\usepackage{cleveref}
\usepackage{float}
\usepackage{changes}
\usepackage{mathrsfs} 
\usepackage{amsmath,latexsym, bbold}
\usepackage{endnotes}
\usepackage{ulem}
\usepackage{aligned-overset}
\usepackage[margin=1in]{geometry}

\allowdisplaybreaks

\newtheorem{thm}{Theorem}[section]

\newtheorem*{maintheorem*}{Main Theorem}
\newtheorem{defn}[thm]{Definition}
\newtheorem{Ex}[thm]{Example} 
\newtheorem{lemma}[thm]{Lemma} 
\newtheorem{proposition}[thm]{Proposition} 
\newtheorem{remark}[thm]{Remark} 
\newtheorem{Cor}[thm]{Corollary}
\newtheorem{ques}[thm]{Question}

\DeclareMathOperator{\Aut}{Aut}
\DeclareMathOperator{\Spec}{Spec}

\newcommand{\cR}{{\mathcal R}^{c}}
\newcommand{\bR}{{\mathcal R}^{b}}

\usepackage{tikz}
\usetikzlibrary{matrix}
\usepackage{graphicx,ctable,booktabs}
\usetikzlibrary{arrows.meta}
\usepackage{tikz-cd}

\newcommand{\kk}{\Bbbk}

\def\GKdim{\operatorname{GKdim}}

\def\it{\textit}

\def\gr{\operatorname {gr}}

\numberwithin{equation}{section}

\title{Relative Cancellation}

\author{Hongdi Huang, Zahra Nazemian,  Yanhua Wang, James J. Zhang}

\address{Huang: Department of Mathematics, Shanghai University, Shanghai, 200444, China}

\email{hdhuang@shu.edu.cn}

\address{Nazemian: University of Graz, Heinrichstrasse 36, 8010 Graz, Austria}

\email{zahra.nazemian@uni-graz.at}

\address{Wang: School of Mathematics, Shanghai Key Laboratory of Financial Information Technology, Shanghai University of Finance and Economics, Shanghai 200433, China}

\email{yhw@mail.shufe.edu.cn}

\address{Zhang: Department of Mathematics, Box 354350,University of Washington, Seattle, Washington 98195, USA}
\email{zhang@math.washington.edu}

\subjclass[2010]{16P99, 16W99} 

\keywords{Cancellation problem, relative cancellative 
property, firm domain, filtration, stable subspace}

\begin{document}

\begin{abstract}
We introduce and study a relative 
cancellation property for associative 
algebras. We also prove a characterization 
result for polynomial rings 
which partially answers a question of Kraft.
\end{abstract}


\maketitle

\section*{Introduction}
\label{xxsec0}

Throughout we assume that $\kk$ is a base 
field and algebras are associative algebras 
over $\kk$, though several definitions work 
for other classes of algebras such as 
Poisson algebras.

The cancellation problem has been an active 
topic in both commutative and 
noncommutative settings, and even in the 
Poisson setting \cite{GWY2022}, see the survey 
papers \cite{Gu3-2015,HTW-2024}. Recall that
an algebra $A$ is called {\sf cancellative} 
if for any algebra $B$, an algebra
isomorphism $A[t] \cong B[t]$ implies that
$A\cong B$.  In the commutative case, the 
famous Zariski Cancellation Problem 
(abbreviated as ZCP) asks if the 
commutative polynomial ring 
$\Bbbk [x_1,\dots,x_n]$ is cancellative for 
all $n\geq 1$. It is well-known that $\Bbbk[x_1]$ 
is cancellative by a result of 
Abhyankar-Eakin-Heinzer \cite{AEH-1972}, 
while $\Bbbk[x_1,x_2]$ is cancellative by 
Fujita \cite{Fu-1979} and Miyanishi-Sugie 
\cite{MS-1980} in characteristic zero and 
by Russell \cite{Ru-1981} in positive
characteristic. The original ZCP (for 
$n\geq 3$) was open for many years. In 2014, 
Gupta \cite{Gu2, Gu1} settled the ZCP negatively
in positive characteristic for $n\geq 3$. The ZCP 
in characteristic zero remains open for $n\geq 3$.
In her ICM invited talk in 2022 and AMS Noether 
lecture in 2025, Gupta presented connections 
between the ZCP and several other important 
questions in affine algebraic geometry. Also 
see an earlier survey paper on related topics 
\cite{Kr-1996}. Our main goal is to introduce 
a new cancellative property which is a 
{\sf cancellative property relative 
to a class of algebras}.

Let ${\mathcal R}$ be a given family of algebras
which will be specified. For example, 
${\mathcal R}$ consists of all universal 
enveloping algebras over finite dimensional 
Lie algebras, or ${\mathcal R}$ is 
\eqref{E0.5.4}, \eqref{E0.7.1}, or 
\eqref{E0.8.1}. 

\begin{defn}
\label{xxdef0.1}
Let ${\mathcal R}$ be a family of algebras. 
\begin{enumerate}
\item[(1)]
An algebra $A$ is called 
{\sf ${\mathcal R}$-cancellative} if, for 
any algebra $B$ and any two algebras 
$R, R'\in {\mathcal R}$ of the same 
Gelfand-Kirillov dimension, an algebra 
isomorphism $A\otimes R\cong B\otimes R'$ 
implies that $A\cong B$. The family of all 
${\mathcal R}$-cancellative algebras is 
denoted by $\cR$. 
\item[(2)]
An algebra $A$ is called 
{\sf ${\mathcal R}$-bicancellative} if, for 
any algebra $B$ and any two algebras 
$R, R' \in {\mathcal R}$ of the same 
Gelfand-Kirillov dimension, an algebra 
isomorphism $A\otimes R\cong B\otimes R'$ 
implies that $A\cong B$ and $R\cong R'$. 
The family of all ${\mathcal R}$-bicancellative 
algebras is denoted by $\bR$.
\item[(3)]
Let ${\mathcal C}$ be another family of 
algebras. We say the pair $({\mathcal R}, 
{\mathcal C})$ is {\sf bicancellative} if 
for any two algebras $C,C'\in {\mathcal C}$ 
and any two algebras $R, R' \in {\mathcal R}$, 
an algebra isomorphism $C\otimes R\cong 
C'\otimes R'$ implies that $C\cong C'$ and 
$R\cong R'$.  
\end{enumerate}
\end{defn}

Note that the algebras $R$ and $R'$ in 
Definition \ref{xxdef0.1}(1) may be 
different, see Example \ref{xxexa1.8}, 
which also shows that $\bR\subsetneq \cR$ 
in general. 

It is clear that cancellative in the 
classical sense is equivalent to 
${\mathcal R}$-cancellative where 
${\mathcal R}$ is taken to be the 
singleton $\{\Bbbk[t]\}$. Strong cancellative 
(defined in \cite[Definition 1.1(b)]{BZ1-2017}) 
means ${\mathcal R}$-cancellative where 
${\mathcal R}$ is taken to be the set of 
all commutative polynomial rings 
$\{\Bbbk[t_1,\cdots,t_n]\}_{n\geq 1}$. 
Universally cancellative (defined in 
\cite[Definition 1.1(c)]{BZ1-2017}) can also 
be expressed as ${\mathcal R}$-cancellative 
when $\mathcal R$ consists of all commutative
domains. 
If ${\mathcal R}$ consists of all 
$\Bbbk$-algebras, then there is no algebra 
that is ${\mathcal R}$-cancellative. If 
${\mathcal R}$ is the singleton $\{\Bbbk\}$, 
then every algebra is ${\mathcal R}$-cancellative.
If ${\overline{\mathcal R}}\subseteq {\mathcal R}$, 
then being ${\mathcal R}$-cancellative is a 
stronger property than being  
$\overline{\mathcal R}$-cancellative. However, 
given a  general family ${\mathcal R}$, it is 
not easy to determine whether $\cR$ (or $\bR$) is 
non-empty. 

In this paper we explore the ${\mathcal R}$-cancellative property 
when ${\mathcal R}$ contains several families of 
well-studied algebras which may be noncommutative.  
For example, ${\mathcal R}$ may consist of the 
following
\begin{enumerate}
\item[($\bullet$)]
the $n$th Weyl algebras for all $n\geq 1$,
\item[($\bullet$)]
the universal enveloping algebra of a finite-dimensional Lie
algebra, including all commutative polynomial 
rings,
\item[($\bullet$)]
the Sklyanin algebras,
\item[($\bullet$)]
tensor products of finitely  many algebras 
mentioned above.
\end{enumerate}

Now we introduce more definitions.

\begin{defn}
\label{xxdef0.2}
An algebra $A$ is called a {\sf firm domain} if 
for every domain $B$, $A\otimes B$ is a domain.
\end{defn}

If $\Bbbk$ is algebraically closed, then every 
commutative domain is a firm domain by Lemma 
\ref{xxlem1.1}(1), see the proof of 
\cite[Lemma 1.1]{RS-2013}. The algebras listed 
before  Definition  \ref{xxdef0.2} are also firm 
domains. But not every division algebra is a firm 
domain \cite[Theorem 4.6]{RS-2013}.

The main tool in this paper is 
${\mathbb N}$-filtration which we recall 
below. Let $A$ be an algebra. An 
${\mathbb N}$-filtration of $A$ is an ascending 
sequence of subspaces 
${\mathbf F}:=\{F_i(A)\subseteq A\}_{i\geq 0}$ 
satisfying 
\begin{enumerate}
\item[(1)]
$\Bbbk \subseteq F_i(A)\subseteq F_{i+1}(A)$ 
for all $i\geq 0$,
\item[(2)]
$A=\bigcup_{i} F_i(A)$,
\item[(3)]
$F_i(A) F_{j}(A)\subseteq F_{i+j} (A)$ for all 
$i,j\geq 0$.
\end{enumerate}
We say a filtration ${\mathbf F}$ is {\sf trivial} 
if $F_0(A)=A$.
Given a filtration ${\mathbf F}$, the associated graded ring is 
defined to be
$$\gr_{\mathbf F}(A)=
\bigoplus_{i=0}^{\infty} F_{i}(A)/F_{i-1}(A),$$
where $F_{-1}(A)=0$ by convention.
We also drop the subscript when the filtration is clear.  

We say 
${\mathbf F}$ is a {\sf good filtration} if 
$\gr A$ is a domain (it is necessary that $A$ 
itself is a domain) and ${\mathbf F}$ is a {\sf firm 
filtration} if $\gr A$ is a firm domain.
Let $P$ be an algebra property. For an 
algebra $A$, set

\begin{equation}
\notag
    \Phi_P (A) := \{F | F\,\, \textnormal{is an}\,\,{\mathbb N}\textnormal{-filtration such that}\, \gr_{\textbf{F}} A \, \textnormal{ satisfies property }\, P\}
\end{equation}

The next set of invariants are dependent on the 
set $\Phi_{P}(A)$.

\begin{defn}
\label{xxdef0.3}
Let $A$ be an algebra and let $P$ be an algebra 
property. The set of {\sf universally degree $0$ 
elements} of $A$ is defined to be
$$u_{P}(A): = \bigcap_{{\mathbf F}\in \Phi_{P}(A)}
F_0(A),$$
where the intersection $\bigcap$ runs over all 
filtrations in  $\Phi_{P}(A)$.
\end{defn}

It is clear that $u_{P}(A)$ is a subalgebra of $A$.

\begin{defn}
\label{xxdef0.4}
Let $A$ be an algebra and let $P$ be an algebra 
property. 
\begin{enumerate}
\item[(1)]
We say $A$ is {\sf $u_P$-maximal} if 
$u_P(A)=A$, or equivalently, $\Phi_{P}(A)$ 
consists of the trivial filtration only. 
\item[(2)]
We say $A$ is {\sf $u_P$-minimal} if $u_{P}(A)=\Bbbk$,
or equivalently,  elements in $A\setminus \Bbbk$ 
are not of universally degree 0. 
\end{enumerate}
\end{defn}

The property $P$ that we consider in this manuscript is being a domain together with possibly additional hypotheses.
The behavior of an algebra with respect to filtrations whose associated graded algebra has this property 
can vary significantly. 

For example, if $P$ is the property of being a domain, then the algebras in Corollary \ref{xxcor1.4} are 
$u_{P}$-maximal and therefore do not admit any nontrivial filtration whose associated graded algebra is a domain. 
On the other hand, all algebras listed before Definition \ref{xxdef0.2} are $u_{P}$-minimal.

Gelfand-Kirillov dimension (GK-dimension) will be 
used as an effective tool. We refer to the books 
\cite{KL-2000, MR-2001} for basic definitions and 
properties of GK-dimension. For two algebras $A$ 
and $B$, $\GKdim (A\otimes B)\leq \GKdim A+\GKdim B$ 
holds in general \cite[Proposition 8.2.3]{MR-2001}. 

\begin{defn}
\label{xxdef0.5}
Let $R$ be an algebra of finite GK-dimension. We 
say $R$ is {\sf $\GKdim$-additive} if for every 
algebra $A$, $\GKdim (A\otimes R)=\GKdim A+\GKdim R$.
\end{defn}

By \cite[Proposition 8.2.7 and Corollary 8.2.15]{MR-2001}, 
the commutative polynomial rings, the Weyl algebras, 
the (skew) Laurent polynomial rings 
$\Bbbk_{\mathbf p}[x_1^{\pm 1},\cdots,x_{n}^{\pm 1}]$, 
and the universal enveloping algebra over a finite 
dimensional Lie algebra  are $\GKdim$-additive. Examples of relative cancellative algebras  with respect to several  important classes are provided in this paper.

Let $firm$ denote the property of being a firm domain 
and $good$ denote the property of being a domain. Let 
$A$ be an algebra. Recall that 
\begin{equation}
\label{E0.5.1}\tag{E0.5.1}
\Phi_{good} (A) := \{F | F\,\, \textnormal{is a good}\,\, \mathbb N\textnormal{-filtration of}\, A\}, 
\end{equation}
and 
\begin{equation}
\label{E0.5.2}\tag{E0.5.2}
\Phi_{firm} (A) := \{F | F\,\, \textnormal{is a firm}\,\, \mathbb N\textnormal{-filtration of}\, A\} 
\end{equation}

By Definition \ref{xxdef0.3}, we have 
\begin{equation}
\label{E0.5.3}\tag{E0.5.3}
u_{good}(A): = \bigcap_{{\mathbf F}\in \Phi_{good}(A)}
F_0(A), \quad {\text{and}} \quad
u_{firm}(A): = \bigcap_{{\mathbf F}\in \Phi_{firm}(A)}
F_0(A).\end{equation}

\subsection*{Class ${\mathcal R}_{1}$:} The first class of relative cancellative algebras  
involves the following family of algebras:
\begin{equation}
\label{E0.5.4}\tag{E0.5.4}
{\mathcal R}_{1} :=
{\text{the family of $\GKdim$-additive 
firm domains that are $u_{firm}$-minimal}.} 
\end{equation}

Note that this ${\mathcal R}_{1}$ contains 
all algebras given before Definition \ref{xxdef0.2}. 

\begin{thm}[Theorem \ref{xxthm2.4}(1,2)]
\label{xxthm0.6}
Let ${\mathcal R}_{1}$ be defined as above. 
Suppose $A$ is a domain of finite GK-dimension 
that is $u_{good}$-maximal.
\begin{enumerate}
\item[(1)]
$A$ is ${\mathcal R}_{1}$-cancellative.
\item[(2)]
If $A$ has an ideal of codimension one, then 
$A$ is ${\mathcal R}_{1}$-bicancellative.
\end{enumerate}
\end{thm}

We have an immediate consequence.

\begin{Cor}
\label{xxcor0.7}
Suppose $\Bbbk$ is algebraically closed.
If $A$ is an affine commutative domain of 
GK-dimension one that is $u_{good}$-maximal
{\rm{(}}or equivalently, not $u_{good}$-minimal{\rm{)}}, 
then $A$ is ${\mathcal R}_{1}$-bicancellative.
\end{Cor}

Note that a domain of GK-dimension one is 
either $u_{good}$-maximal or $u_{good}$-minimal 
[Proposition \ref{xxpro1.6}]. 

\subsection*{Class ${\mathcal R}_{2}$: }
The second class of our examples involves the following 
class of algebras
\begin{equation}
\label{E0.7.1}\tag{E0.7.1}
{\mathcal R}_{2} :=
{\text{the family of $\GKdim$-additive 
domains that are $u_{good}$-minimal}.}
\end{equation}

So ${\mathcal R}_{2}$ is strictly bigger than 
${\mathcal R}_{1}$. Therefore 
${\mathcal R}_{2}$-cancellative algebras are 
${\mathcal R}_{1}$-cancellative (but not the 
converse). Note that if $A$ is a commutative 
domain over an algebraically closed field, 
then every good filtration of $A$ is a firm 
filtration. 

\begin{thm}[Theorem \ref{xxthm2.5}(1,2)]
\label{xxthm0.8}
Let ${\mathcal R}_{2}$ be defined as in 
\eqref{E0.7.1}. Suppose $A$ is a 
$u_{good}$-maximal firm domain of finite 
GK-dimension.
\begin{enumerate}
\item[(1)]
$A$ is ${\mathcal R}_{2}$-cancellative.
\item[(2)]
If $A$ has an ideal of codimension one, then 
$A$ is ${\mathcal R}_{2}$-bicancellative.
\end{enumerate}
\end{thm}

\subsection*{ Classes ${\mathcal R}_{3}$ and ${\mathcal C}_{3}$:}
Consider the class
\begin{equation}
\label{E0.8.1}\tag{E0.8.1}
{\mathcal R}_{3} :=
{\text{the family of affine commutative 
domains that are $u_{good}$-maximal,}}
\end{equation}
and 
\begin{equation}
\label{E0.8.2}\tag{E0.8.2}
{\mathcal C}_{3} :=
{\text{the family of algebras whose centers 
are $u_{good}$-minimal domains.}}
\end{equation}

\begin{thm}[Theorem \ref{xxthm2.6}(1)]
\label{xxthm0.9}
Suppose $\Bbbk$ is algebraically closed.
Let ${\mathcal R}_{3}$ and ${\mathcal C}_{3}$ 
be defined as above. Then $({\mathcal R}_{3}, 
{\mathcal C}_{3})$ is a  bicancellative pair.
\end{thm}

\subsection*{Characterization problem}
 In his 
Bourbaki seminar, Kraft \cite{Kr-1996} 
considered the following problem as one of 
the eight challenging problems in affine 
algebraic geometry (along with Jacobian 
Conjecture, Automorphism problem, the 
ZCP, etc).

\begin{ques}[Characterization Problem] 
\label{xxque0.10}
Find an algebraic-geometric characterization 
of $\Bbbk[z_1,\cdots,z_m]$.
\end{ques}

Let $V$ be a subspace of $A$. We say $V$ is 
{\sf $\Aut$-stable} if $\sigma(V)\subseteq V$ 
for all algebra automorphisms $\sigma\in 
\Aut_{alg}(A)$. Every algebra $A$ has three 
obvious $\Aut$-stable subspaces, namely, $A$, 
$0$, and $\kk$, which are called {\sf trivial} 
$\Aut$-stable subspaces of $A$.
The characterization and
cancellation problems are related through $u_P (A)$, which is Aut-stable, [Lemma \ref{xxlem3.1}]. The next result says 
that the polynomial rings and the Weyl algebras 
are special with respect to the $\Aut$-stable 
property.

\begin{thm}
\label{xxthm0.11}
Suppose ${\rm{char}}\; \Bbbk=0$. Let $A$ be 
either the polynomial ring $\Bbbk[z_1,\cdots,z_m]$ 
for $m\geq 2$ or the $n$th Weyl algebra 
${\mathbb A}_{n}$ for $n\geq 1$. Then $A$ has 
no nontrivial $\Aut$-stable subspaces.
\end{thm}

Motivated by the above result, we prove a 
characterization of the polynomial rings which 
partially answers Question \ref{xxque0.10}.

\begin{thm}[Theorem \ref{xxthm3.6}]
\label{xxthm0.12}
Suppose $\Bbbk$ is an algebraically closed field 
of characteristic zero. Let $A$ be an algebra 
{\rm{(}}which is not-necessarily-commutative{\rm{)}}. 
Then $A$ is ismorphic to $\Bbbk[z_1,\cdots,z_m]$ 
for some $m\geq 2$ if and only if the following 
two conditions hold:
\begin{enumerate}
\item[(1)]
$A\neq \Bbbk$ is finitely generated {\rm{(}}namely, affine{\rm{)}} and 
connected graded. 
\item[(2)]
$A$ does not have any nontrivial $\Aut$-stable 
subspaces. 
\end{enumerate}
\end{thm}

Condition (1) in the above theorem can not be removed since the Weyl algebras satisfy condition (2). However,
we are wondering if condition (1) can be replaced by another weaker condition.

\section{Preliminaries}
\label{xxsec1}

In this section we collect several basic facts 
about firm domains and filtrations that will 
be used throughout the paper. As a consequence, 
we provide examples of $u_P$-maximal and 
$u_P$-minimal algebras. When $\Bbbk$ is an 
algebraically closed field of characteristic 
$p>0$, we introduce a class ${\mathcal R}$ 
of $\Bbbk$-algebras that includes all 
polynomial algebras, and identify certain 
${\mathcal R}$-cancellative algebras.
We start with the following lemma:

\begin{lemma}
\label{xxlem1.1}
\begin{enumerate}
\item[(1)]\cite[Lemma 1.1]{RS-2013}
Suppose $\Bbbk$ is algebraically closed.
Commutative domains are firm domains.
\item[(2)]
If $A$ is a firm domain and $B$ is a 
subalgebra of $A$, then $B$ is a firm domain.
\item[(3)]
Suppose $A$ is a firm domain. Then every Ore 
extension of $A$ is a firm domain.
\item[(4)]
The tensor product of two firm domains is a firm domain. 
\item[(5)]
Suppose $\Bbbk$ is algebraically closed.
Every connected graded domain of GK-dimension 
two is a firm domain.
\item[(6)]
Suppose $B$ is a subalgebra of $A$. If $A$ is 
a domain {\rm{(}}resp. firm domain{\rm{)}}, 
then $u_{good}(B)\subseteq u_{good}(A)$ 
{\rm{(}}resp. $u_{firm}(B)\subseteq u_{firm}(A)${\rm{)}}.
\end{enumerate}
\end{lemma}

\begin{proof}
(2,3,4) Clear.

(5) Let $Q$ be the graded quotient ring of $A$. By 
\cite[Theorem 0.1]{AS-1995}, $Q$ is of the form 
$K[t^{\pm 1};\sigma]$ where $K$ is a field of 
GK-dimension one. By part (1), $K$ is a firm domain. 
By part (3), $Q$ is a firm domain. Since $A$ is a 
subalgebra of $Q$, $A$ is a firm domain by part (2). 

(6) This is a consequence of the fact that every 
good {\rm{(}}resp. firm{\rm{)}} filtration on 
$A$ restricts to a good {\rm{(}}resp. 
firm{\rm{)}}  filtration on $B$.
\end{proof}

Let ${\mathbf F}$ be a filtration of $A$ and 
${\mathbf G}$ be a filtration of $B$. We 
define a filtration, called ${\mathbf H}$, 
of the tensor product $A\otimes B$ as follows:
\begin{equation}
\label{E1.1.1}\tag{E1.1.1}
H_i(A\otimes B):=\sum_{j=0}^{i} F_{j}(A)
\otimes G_{i-j}(B) \subseteq A\otimes B.
\end{equation}
The following lemma is well-known.

\begin{lemma}
\label{xxlem1.2}
Let $A$ and $B$ be two algebras. Let 
${\mathbf F}$ {\rm{(}}resp. 
${\mathbf G}${\rm{)}} be a filtration of 
$A$ {\rm{(}}resp. $B${\rm{)}}.
\begin{enumerate}
\item[(1)]
${\mathbf H}$ defined in \eqref{E1.1.1} is 
a filtration of $A\otimes B$.
\item[(2)]
$\gr_{\mathbf H} (A\otimes B)
\cong \gr_{\mathbf F} A\otimes \gr_{\mathbf G} B$.
\item[(3)]
If ${\mathbf F}$ is a good filtration and ${\mathbf G}$ is a 
firm filtration, then ${\mathbf H}$ is a good filtration.
\item[(4)]
If both ${\mathbf F}$ and ${\mathbf G}$ are firm filtrations, then 
so is ${\mathbf H}$.
\end{enumerate}
\end{lemma}

For the rest of this paper, $P$ is the property of being a domain
plus something else to be specified.
This means that we only consider good filtrations 
of $A$. If ${\mathbf F}$ is a good filtration of 
$A$, then for all nonzero elements $x,y\in A$,
\begin{equation}
\label{E1.2.1}\tag{E1.2.1}
\gr_{\mathbf F}(xy)=\gr_{\mathbf F}(x)
\gr_{\mathbf F}(y),
\end{equation}
where $\gr_{\mathbf F}(x)$ denotes the quotient
class of $x$ in the quotient space 
$F_{i}(A)/F_{i-1}(A)
(=\gr_{\mathbf F}(A)_i)$ if $x\in F_i(A)
\setminus F_{i-1}(A)$. If ${\mathbf F}$ is 
understood, then we use $\gr$ instead of 
$\gr_{\mathbf F}$.

We say an algebra $A$ is (left) {\sf integral 
over} a subalgebra $D\subseteq A$ if for 
every element $f\in A$, there are elements 
$d_s\in D$, not all zero, such that 
$\sum_{s=0}^{n} d_s f^s=0$. We implicitly 
assume that $\Phi_{P}(A)\neq \emptyset$, and 
a filtration always means a filtration in 
$\Phi_{P}(A)$.

\begin{proposition}  
\label{xxpro1.3}  
Let $A$ be a domain that is integral over 
$u_P(A)$. Then $A$ is $u_P$-maximal.  
\end{proposition}  

\begin{proof}  
We claim that $A = u_P(A)$. Suppose to the 
contradiction that there exists an element 
$f \in A \setminus u_P(A)$. Then there 
exists a filtration ${\mathbf F}$ such that 
$f \not\in F_0(A)$. Consequently, the 
sequence $\{f^n\}_{n \geq 1}$ consists of 
elements of different degrees, by \eqref{E1.2.1}.  

Since $A$ is integral over $u_{P}(A)$,   
$\sum_{i = 0}^{n} d_i f^i = 0$, 
where $d_n \neq 0$ and all $d_i$ are in $u_P(A)$.  
Without loss of generality, we may assume that 
$f\in F_m\setminus F_{m-1}$ for some $m\geq 1$. 
Then $d_nf^n\in F_{nm}\setminus F_{mn-1}$ by 
\eqref{E1.2.1}. Similarly, we have 
$d_if^i\in F_{nm-1}$ for all $i<n$. This 
implies that $d_nf^n=-\sum_{i=0}^{n-1}d_if^i
\in F_{mn-1}$, giving a contradiction.
\end{proof}  

\begin{Cor}
\label{xxcor1.4}
Let $A$ be a domain, and let $D$ be the 
subalgebra of $A$ generated by all 
invertible elements in $A$. If $A$ is 
integral over $D$, then $A$ is $u_P$-maximal. 
As a consequence, every division algebra 
is $u_P$-maximal.  
\end{Cor}  

\begin{proof}  
We cliam that $D \subseteq u_P(A)$. Suppose to 
the contrary that $D\not\subseteq u_{P}(A)$.
Then there is an invertible element 
$x \in A$ such that $x \not\in F_0(A)$ for 
some good filtration ${\mathbf F}$. Then 
$\gr(x)$ has 
positive degree, while $\gr(x^{-1})$ has 
non-negative degree. However, by \eqref{E1.2.1}, 
we have  
\[
1 = \gr(1) = \gr(x) \gr(x^{-1}),
\]  
leading to a contradiction.  

Since $A$ is integral over $D$, it is also 
integral over $u_P(A)$. By Proposition 
\ref{xxpro1.3}, $A$ is $u_{P}$-maximal.
The consequence is clear.
\end{proof}  
 
\begin{lemma}
\label{xxlem1.5} 
Let ${\mathbf F}$ be a good filtration of 
an algebra $Y$. Let $X$ be a subalgebra of 
$F_0(Y)$. If $Z$ is a subalgebra of $Y$ 
containing $X$ such that $\GKdim Z<\GKdim 
X+1<\infty$, then $Z\subseteq F_0(Y)$. 
\end{lemma}

\begin{proof} Suppose that $Z$ is not a 
subset of $F_0(Y)$. Then there is an 
element $z\in Z$ not in $F_0(Y)$. Then 
$z^n$ has different degrees for different 
$n$. Therefore, $\sum_{n\geq 0} z^n X$ is 
a direct sum. This implies that 
$\GKdim Z\geq \GKdim X+1$, yielding a 
contradiction. 
\end{proof}

\begin{proposition}
\label{xxpro1.6}
Suppose $\Bbbk$ is algebraically closed.
Let $A$ be a domain of GK-dimension one. 
Then one of the following holds:
\begin{enumerate}
\item[(1)]
$u_{P}(A)=A$.
Namely, $A$ is $u_P$-maximal. 
\item[(2)]
$u_{P}(A)=\Bbbk$. Namely, $A$ is $u_P$-minimal.
\end{enumerate}
\end{proposition}

Proposition \ref{xxpro1.6} indicates that 
generically a domain of GK-dimension one is 
$u_P$-maximal.

\begin{proof}[Proof of Proposition \ref{xxpro1.6}]
Suppose $X:=u_{P}(A)\neq A$. It remains to show 
that $X=\Bbbk$. If not, $\GKdim X=1$. Let 
${\mathbf F}$ be  any non-trivial filtration 
of $A$. Since $\GKdim A=\GKdim X$ and 
$X\subseteq F_0(A)$, by Lemma \ref{xxlem1.5}, 
$A\subseteq F_0(A)$, yielding a contradiction.
\end{proof}

\begin{proposition}
\label{xxpro1.7} 
Suppose $\Bbbk$ is algebraically closed. Let 
$A$ be a domain of GK-dimension two. Then one 
of the following holds:
\begin{enumerate}
\item[(1)]
$A$ is $u_P$-maximal.
\item[(2)]
$A$ is $u_P$-minimal.
\item[(3)]
$u_{P}(A)$ is a commutative domain of 
GK-dimension one. Further, for every 
nontrivial filtration ${\mathbf F}$ of $A$, 
$F_0(A)=u_{P}(A)$.
\end{enumerate}
\end{proposition}

\begin{proof} 
Suppose $A$ is
not $u_P$-maximal.  Then take   a 
nontrivial filtration
${\mathbf F}$ of $A$. 
Then  $F_0(A)\neq A$. Since $A$ is a domain and ${\mathbf F}$ is a good filtration, we have 
$\gr_{\mathbf F}(A) = F_0(A) \oplus I$, 
where 
$I := \bigoplus_{i \geq 0} F_{i+1}(A)/F_i(A)$ 
is also a domain. 
Because $I \neq 0$, it contains a nonzero (and hence regular) element of $\gr_{\mathbf F}(A)$, and by \cite[Proposition~3.15]{KL-2000} it follows that 
\[
\GKdim \gr_{\mathbf F}(A) \geq \GKdim \big(\gr_{\mathbf F}(A)/I\big) + 1.
\]
Using the facts that 
$\GKdim A \geq \GKdim \gr_{\mathbf F}(A)$ (see \cite[Lemma~6.5]{KL-2000}) 
and 
$\GKdim F_0(A) = \GKdim \big(\gr_{\mathbf F}(A)/I\big)$, 
we obtain
\[
\GKdim A \geq \GKdim F_0(A) + 1.
\]

Thus $F_0(A)$ is a domain of GK-dimension 
at most one. Since $\Bbbk$ is algebraically 
closed, $F_0(A)$  is commutative. Since $u_{P}(A)$ 
is a subalgebra of $F_0(A)$, it is also a commutative domain.

If $\GKdim u_{P}(A)=0$, then $u_{P}(A)=\Bbbk$, 
which is Case (2). Otherwise, $u_{P}(A)$ is 
a commutative algebra of GK-dimension one; that is, it is in Case (3). 
Since $  X:=u_{P}(A)  \subseteq F_0(A) $, we have that 
$\GKdim F_0(A) =  1$. 
 Now consider another arbitrary
nontrivial filtration ${\mathbf G}$ of $A$.
By a similar argument, we have that $\GKdim G_0(A) =  1$.   Therefore,  $Z:=F_0(A)$ is a subalgebra of 
$Y:=A$ containing $X$ such that $\GKdim X=
\GKdim Z$. Now since 
$X \subseteq G_0 (A)$, 
it follows from Lemma \ref{xxlem1.5} that
$ Z =  F_0(A)\subseteq G_0(A)$. Since the filtration 
${\mathbf G}$ of $A$ is arbitrary, 
$ F_0(A) \subseteq u_{P}(A)$, and consequently, 
$F_0(A) = u_{P}(A)$. 
\end{proof}

Before we prove our main results, we give 
an example.

\begin{Ex}
\label{xxexa1.8}
{\rm 
Let $\Bbbk$ be an algebraically closed 
field of characteristic $p>0$ and let $A$ 
be the algebra
\begin{equation}
\label{E1.8.1}\tag{E1.8.1}
\Bbbk[X,Y,Z,T]/(X^{m}Y+Z^{p^e}+T+T^{sp})
\end{equation}
where $m>1, e, s$ are positive integers 
such that $p^e\nmid sp$ and $sp\nmid p^e$. 
By a result of Asanuma \cite[Theorem 5.1 and 
Corollary 5.3]{As-1987}, also see 
\cite[Theorem 1.1]{Gu3-2015},
$ A \otimes \kk [t] \cong A[t]\cong \Bbbk[t_1,t_2,t_3,t_4]$.
Gupta proved that $A\not\cong 
\Bbbk[t_1,t_2,t_3]$ \cite[Theorem 3.8]{Gu3-2015}, 
so this provides a counterexample to the ZCP 
in positive characteristic case.  In this example, we see that this algebra is cancellative with respect to 
affine domains of GK-dimension one that are not polynomial. Let us explain further, 
and let ${\mathcal R}$ denote the class consisting of all 
polynomial rings and the algebras of the form~\eqref{E1.8.1}.


\medskip
\noindent
Claim 1: Every affine domain over $\Bbbk$ 
of GK-dimension one is 
${\mathcal R}$-cancellative.

\noindent
{\it{ Proof of  Claim 1:}}
Let $A$ be an affine domain of GK-dimension 
one   and let $B$ be another algebra. Suppose
$A\otimes R\cong B\otimes R'$ for $R,R'$ 
in ${\mathcal R}$ of the same GK-dimension, 
say $n$. Then $R$ {\rm{(}}or $R'${\rm{)}} is 
either the polynomial ring or an algebra in 
\eqref{E1.8.1}. So both $R[t]$ and $R'[t]$ 
are polynomial rings by the property of 
algebras in \eqref{E1.8.1}. So 
$$A\otimes \Bbbk[t_1,\cdots,t_{n+1}]
\cong A\otimes R\otimes \Bbbk[t]
\cong B\otimes R'\otimes \Bbbk[t]
\cong B\otimes \Bbbk[t_1,\cdots,t_{n+1}].$$ 
By \cite[Corollary 3.4]{AEH-1972}, 
$A$ is cancellative (= invariant). Consequently, 
$A\cong B$. So $A \in 
\cR$. In particular,
$\Bbbk[t]\in \cR$.  

\medskip

\noindent
Claim 2: $\Bbbk[t]\not\in \bR$.

\noindent
{\it {Proof of Claim 2:}}
When $A=B=\Bbbk[x]$, one may take 
$R=\Bbbk[t_1,t_2,t_3]$ and $R'$ be the 
algebra defined in \eqref{E1.8.1}, then 
$A\otimes R\cong A\otimes R'$ for $R\not\cong R'$. 
This also shows that $\Bbbk[t]\not\in \bR$. 
As a consequence, $\bR\subsetneq \cR$. 

\medskip

\noindent
Claim 3: Let ${\mathcal C}$ be the set of 
affine domain of GK-dimension one that 
are not the polynomial ring $\Bbbk[t]$. 
Then $({\mathcal R}, {\mathcal C})$ is a  
bicancellative pair.

\noindent
{\it {Proof of Claim 3:}}
Let $A, A'$ be in ${\mathcal C}$ and 
$R, R'\in {\mathcal R}$. If $A\otimes R
\cong A'\otimes R'$, then $\GKdim A=
\GKdim A'=1$ implies that
$\GKdim R=\GKdim R'=:n$. Let $\phi: 
A\otimes R\to A'\otimes R'$ be an 
isomorphism of algebras. Let $\psi$ be 
the isomorphism $\phi\otimes Id_{\Bbbk[t]}:
A\otimes R\otimes \Bbbk[t]\to 
A'\otimes R'\otimes \Bbbk[t]$. Since both 
$R\otimes \Bbbk[t]$ and $R'\otimes \Bbbk[t]$ 
are isomorphic to $\Bbbk[t_1,\cdots,t_{n+1}]$, 
by the strongly invariant property of $A$ 
(see \cite[Corollary 3.4]{AEH-1972}), 
$\psi(A)=A'$. By restriction $\phi(A)=A'$. 

Let $I$ be an ideal of $A$ that has codimension one. Then $I':=\phi(I)$ is an 
ideal of $A'$ that has codimension one. Hence
\[R\cong \Bbbk \otimes R\cong A/I\otimes R \cong (A\otimes R)/(I) \xrightarrow{\phi/(I)} (A'\otimes R')/(I')\cong A'/I'\otimes R' \cong R'.\] So $(\mathcal {R}, \mathcal {C})$ is a bicancellative pair. }
\end{Ex}



\section{Relative Cancellation problem}
\label{xxsec2}
In this section we prove the main theorems 
concerning the relative cancellation problem. 
We refer to papers \cite{Kr-1996, Gu3-2015, 
BZ1-2017, LWZ2019, HTW-2024} for the famous 
Zariski cancellation problem (ZCP) and some 
variations of cancellation problem for 
noncommutative algebras.
An  ${\mathbb N}$-graded $\kk$-algebra 
$ A = \bigoplus _{i \geq 0} A_i $ is called 
connected if $A_0 = \kk$. 
	Recall that a connected graded algebra $A$ is called a {\sf  $d$-dimensional Artin-Schelter regular 
		(AS-regular) algebra}  
	if $A$ satisfies the following conditions: 
	\begin{itemize}
		\item [(i)] ${\rm gldim}\,A=d<\infty$, 
		\item [(ii)] ${\rm GKdim}\,A<\infty$,  
		\item [(iii) ] 
		${\rm Ext}_{A}^{i}(\kk ,A)\cong \left\{
		\begin{array}{ll}
		k   &\quad (i=d), \\
		0   &\quad (i\neq d). 
		\end{array}
		\right.$
	\end{itemize}

 \begin{lemma}
\label{xxlem2.1} 
For a $\kk$-algebra $A$, the following statements hold:
\begin{enumerate}
\item[(1)]
Suppose $A$ has a firm filtration ${\mathbf F}$. 
Then $A$ is a firm domain.
\item[(2)]
The tensor product of two $\GKdim$-additive 
algebras is $\GKdim$-additive.
\item[(3)]
Let $A$ be a noetherian AS regular algebra. 
Then $A$ is $\GKdim$-additive. 
\end{enumerate}
\end{lemma}

\begin{proof}
(1) Let $B$ be another algebra. Then 
$A\otimes B$ has a filtration defined by 
$F_{i}(A\otimes B):=F_{i}(A)\otimes B$. Then 
$\gr (A\otimes B)\cong (\gr A)\otimes B$. 
If $B$ is a domain, so is $(\gr A)\otimes B$ 
since $\gr A$ is a firm domain. By lifting, 
$A\otimes B$ is a domain. So $A$ is a firm 
domain by definition. 

(2) Clear.

(3) There is a subspace of generators $V$ 
containing $1$ such that $\dim V^{n}$ is a 
multi-polynomial in $n$ of degree 
$d=\GKdim A$ for $n\gg 0$. So $A$ is 
$\GKdim$-additive by 
\cite[Lemma 8.2.4]{MR-2001}.
\end{proof}

Well-known examples of noetherian AS-regular algebras are Sklyanin algebras. We refer the reader to \cite[Section 4, Construction 4.1]{TV-1996} for the definition.
\medskip

\begin{lemma}
\label{xxlem2.2} 
Suppose ${\rm{char}}\; \Bbbk=0$. 
Then the following algebras are in 
${\mathcal R}_{1}$ \eqref{E0.5.4}.
\begin{enumerate}
\item[(a)]
The $n$th Weyl algebras for all $n\geq 1$.
\item[(b)]
The universal enveloping algebras over 
finite dimensional Lie algebras, including 
all commutative polynomial rings,
\item[(c)]
Every noetherian AS regular algebra that 
is a firm domain. In particular, all Sklyanin algebras, provided that $\Bbbk$ is algebraically closed.
\end{enumerate}
\end{lemma}

\begin{proof}
(a,b) By \cite[Proposition 8.2.7]{MR-2001}, the 
Weyl algebras and the universal enveloping 
algebras over  finite dimensional Lie algebras 
are $\GKdim$-additive. For each algebra $A$ 
in (a) or (b), there is a filtration ${\mathbf F}$ 
such that $\gr_{\mathbf F} A$ is isomorphic to 
the commutative polynomial ring. It is well-known 
that the commutative polynomial rings are 
firm domains. By Lemma \ref{xxlem2.1}(1), $A$ is 
a firm domain. It is well-known that if $A$ is in 
either (a) or (b), there is a 
filtration ${\mathbf F}$ such that 
$F_0(A)=\Bbbk$. Hence $u_{P}(A)=\Bbbk$. 

(c) Let $A$ be a  noetherian AS-regular algebra. Then it  is $\GKdim$-additive by Part 3 of Lemma \ref{xxlem2.1}. 
Since $A$ is connected graded, if it is a firm domain, then it is $u_{\rm firm}$-minimal; that is, $A \in {\mathcal R}_{1}$.  
If $S_n$ is a Sklyanin algebra of dimension $n$, it can be considered as a subalgebra of $T_n$, 
which is an iterated Ore extension of a commutative domain $K$, see \cite[Appendix A]{TV-1996}. 
Since $K$ is a firm domain [Lemma \ref{xxlem1.1}(1)], so is $T_n$ by Lemma \ref{xxlem1.1}(3). 
Finally, since $S_n$ is a subalgebra of $T_n$, $S_n$ is also a firm domain by Lemma \ref{xxlem1.1}(2).

\end{proof}

Let $A$ be an algebra. Recall that 
$\Phi_{good}(A)$, $\Phi_{firm}(A)$, $u_{good}(A)$, 
and $u_{firm}(A)$ are defined in 
\eqref{E0.5.1}-\eqref{E0.5.3}. In the next 
few propositions we may identify a subalgebra 
of $A$, say $A'$, with the subalgebra 
$A'\otimes \Bbbk$ inside $A\otimes B$.

\begin{proposition}
\label{xxpro2.3} 
Let $A$ and $B$ be two domains.
\begin{enumerate}
\item[(1)]
$u_{good}(A)\otimes u_{good}(B)
\subseteq u_{good}(A\otimes B)
\subseteq u_{good}(A)\otimes u_{firm}(B)$.
\item[(2)]
 $u_{firm}(A\otimes B)= 
u_{firm}(A)\otimes u_{firm}(B)$.
\item[(3)]
Suppose that $B$ is $u_{firm}$-minimal. Then 
$u_{good}(A\otimes B)= u_{good}(A)$.
\end{enumerate}
\end{proposition}

\begin{proof}
(1) Since $A$ is a subalgebra of $A\otimes B$, 
we have $u_{good}(A)\subseteq 
u_{good}(A\otimes B)$. Similarly, 
$u_{good}(B)\subseteq u_{good}(A\otimes B)$. 
Therefore $u_{good}(A)\otimes u_{good}(B)
\subseteq u_{good}(A\otimes B)$.

For $u_{good}(A\otimes B)\subseteq 
u_{good}(A)\otimes u_{firm}(B)$, let 
${\mathbf F}$ be a good filtration of $A$ 
and ${\mathbf G}$ be a firm filtration of $B$. 
By Lemma \ref{xxlem1.2}(3), 
${\mathbf H}:={\mathbf F}\otimes {\mathbf G}$ 
is a good filtration of $A\otimes B$ with 
$H_0=F_0(A)\otimes G_0(B)$. So 
$$\begin{aligned}
u_{good}(A\otimes B)&=\bigcap_{{\mathbf H}
\in \Phi_{good}(A\otimes B)} H_0
\subseteq 
\bigcap_{{\mathbf H}={\mathbf F}\otimes {\mathbf G},
{\mathbf F}\in 
\Phi_{good}(A), {\mathbf G}\in \Phi_{firm}(B)} 
F_0(A)\otimes G_0(B)\\
&=\Bigl[\bigcap_{{\mathbf F}\in \Phi_{good}(A)} F_0(A)\Bigr]
\otimes 
\Bigl[\bigcap_{{\mathbf G}\in \Phi_{firm}(B)} G_0(B)\Bigr]
=u_{good}(A)\otimes u_{firm}(B).
\end{aligned}
$$

(2) The proof is similar to the proof of
part (1), and we use Lemma \ref{xxlem1.2}(4) 
for it.

(3) This is an immediate consequence of part (1).
\end{proof}

The group of all $\Bbbk$-algebra automorphisms 
of $A$ is denoted by $\Aut_{alg}(A)$.
If $B$ is a subalgebra of $A$, then we 
set
$$\Aut_{alg}(A\mid B):=
\{\sigma\in \Aut_{alg}(A) \mid \sigma(b)=b 
{\text{ for all $b\in B$}}\}.
$$
Now we prove our main results.

\begin{thm}
\label{xxthm2.4}
Let ${\mathcal R}_{1}$ be defined as in 
\eqref{E0.5.4}. Suppose that $A$ is a 
$u_{good}$-maximal domain of finite 
GK-dimension. Then the following hold:
\begin{enumerate}
\item[(1)]
$A$ is ${\mathcal R}_{1}$-cancellative.
\item[(2)]
If $A$ has an ideal of codimension one
{\rm{(}}for example, $A$ is affine and 
commutative over an algebraically closed 
field{\rm{)}}, then $A$ is 
${\mathcal R}_{1}$-bicancellative.
\item[(3)]
If $R\in {\mathcal R}_{1}$, then the 
automorphism group of $A\otimes R$ fits 
into the following short exact sequence
$$1\to \Aut_{alg}(A\otimes R\mid A)\to 
\Aut_{alg}(A\otimes R)\to \Aut_{alg}(A)\to 1.$$
\end{enumerate}
\end{thm}

\begin{proof}
(1) Suppose that $B$ is an algebra such 
that $\phi:A\otimes R\to  B\otimes R'$ is 
an isomorphism where $R,R'\in {\mathcal R}_{1}$ 
are of the same GK-dimension. Since $R$ and 
$R'$ are $\GKdim$-additive, we have 
$$ \GKdim A=\GKdim (A\otimes R)-\GKdim R
=\GKdim (B\otimes R')-\GKdim R'=\GKdim B.$$
We identify $A$ (resp. $B$) with the 
subalgebra $A\otimes \Bbbk$ inside $A\otimes R$
(resp. $B\otimes \Bbbk$ inside $B\otimes R'$).
Since $A$ is $u_{good}$-maximal, we have
$$\begin{aligned}
A&= u_{good}(A) =u_{good}(A\otimes R)\qquad \;  \qquad 
{\text{by Proposition \ref{xxpro2.3}(3)}}\\
&\cong u_{good}(B\otimes R') \qquad \qquad \quad 
\qquad \qquad 
{\text{via the isomorphism $\phi$}}\\
& =  u_{good}(B) \qquad \qquad \qquad \; \quad 
\qquad \qquad 
{\text{by Proposition \ref{xxpro2.3}(3)}}\\
&\subseteq B.
\end{aligned}$$
So $\phi$ maps $A$ into $B$. It remains to show 
that $\phi\mid_{A}$ is surjective. To apply Lemma \ref{xxlem1.5} we consider a 
good filtration ${\mathbf H}$ on $Y:=
A\otimes R\cong B\otimes R'$, and let 
$X:=A$. Then $X$ is a 
subalgebra of $H_0(Y)$. Let $Z:=\phi^{-1}(B)$. 
Then $Z$ is a subalgebra of $Y$ containing 
$X(\cong u_{good}(Y))$ such that 
$\GKdim Z=\GKdim B= \GKdim X$. By 
Lemma \ref{xxlem1.5}, $Z\subseteq H_0(Y)$. 
This means that $Z\subseteq u_{good}(Y)=X$ 
by taking the intersection over all good 
filtrations ${\mathbf H}$. So 
$\phi^{-1}(B)=Z=X$ as required.

(2) It remains to show that $R\cong R'$. Let 
$I$ be an ideal of $A$ of codimension one. 
Then $A/I\cong \Bbbk$. Since $\phi$ maps $A$ 
to $B$ as an isomorphism by part (1), 
$J:=\phi(I)$ is an ideal of $B$ of codimension 
one. Thus
$$R\cong \Bbbk\otimes R\cong (A/I)\otimes R
\cong 
(A\otimes R)/(I)\xrightarrow{\phi \; \cong} 
(B\otimes R')/(J)\cong (B/J)\otimes R'\cong R'$$
as required.

(3) In the setting of the proof of part (1), 
an algebra isomorphism $\phi: A\otimes R\to 
B\otimes R'$ restricts to an algebra 
isomorphism $\phi\mid_{A}: A\to B$. Now, let $A=B$
and $R=R'$. Then every algebra automorphism 
$\phi\in \Aut_{alg}(A\otimes R)$ restricts to 
an algebra automorphism $\phi\mid_{A}\in 
\Aut_{alg}(A)$. Hence $\Psi:\phi\to\phi\mid_{A}$ 
is a group homomorphism from 
$\Aut_{alg}(A\otimes R)\to \Aut_{alg}(A)$.
For every $\sigma\in \Aut_{alg} (A)$, 
$\phi:=\sigma\otimes Id_{R}$ is an 
automorphism of $A\otimes R$ such that 
$\sigma=\phi\mid_{A}$. So $\Psi$ is surjective. 
It is clear that the kernel of $\Psi$ is 
$\Aut_{alg}(A\otimes R\mid A)$. Therefore 
the short exact sequence follows.
\end{proof}

Similarly we have 

\begin{thm}
\label{xxthm2.5}
Let ${\mathcal R}_{2}$ be defined as in 
\eqref{E0.7.1}. Suppose $A$ is a firm domain 
of finite GK-dimension that is $u_{good}$-maximal. 
Then the following hold:
\begin{enumerate}
\item[(1)] 
$A$ is ${\mathcal R}_{2}$-cancellative.
\item[(2)]
If $A$ has an ideal of codimension one 
{\rm{(}}for example, $A$ is affine and 
commutative over an algebraically closed 
field{\rm{)}}, then $A$ is 
${\mathcal R}_{2}$-bicancellative.
\item[(3)]
If $R\in {\mathcal R}_{2}$, then the automorphism 
group of $A\otimes R$ fits into the short 
exact sequence
$$1\to \Aut_{alg}(A\otimes R\mid A)\to 
\Aut_{alg}(A\otimes R)\to \Aut_{alg}(A)\to 1.$$
\end{enumerate}
\end{thm}

The proof of Theorem \ref{xxthm2.5} is omitted 
since it is similar to the proof of Theorem \ref{xxthm2.4}.  

%
%

\begin{proof}[Proof of Corollary \ref{xxcor0.7}] 
By Proposition \ref{xxpro1.6}, $A$
is $u_{good}$-maximal.
Since $A$ is affine and commutative, it has an ideal of codimension one. The assertion 
follows from Theorem \ref{xxthm2.4}(2).
\end{proof}

Recall that ${\mathcal R}_{3}$ and ${\mathcal C}_{3}$ 
are defined as in \eqref{E0.8.1}-\eqref{E0.8.2}.

\begin{thm}
\label{xxthm2.6}
Suppose the base field $\Bbbk$ is algebraically closed. 
Then the following hold:
\begin{enumerate}
\item[(1)]
$({\mathcal R}_{3}, {\mathcal C}_{3})$ is a 
bicancellative pair.
\item[(2)]
If $R\in {\mathcal R}_{3}$ and $A\in {\mathcal C}_{3}$, 
then the automorphism group of $A\otimes R$ fits 
into the short exact sequence
$$1\to \Aut_{alg}(A\otimes R\mid R)\to 
\Aut_{alg}(A\otimes R)\to \Aut_{alg}(R)\to 1.$$
\end{enumerate}
\end{thm}

\begin{proof}
(1) Let $R,R'$ be in ${\mathcal R}_{3}$
and $A, A'$ be in ${\mathcal C}_{3}$ such 
that $\phi: A\otimes R\to A'\otimes R'$
is an isomorphism. 
Let $C=Z(A)$ and $C'=Z(A')$. Then $\phi$ 
restricts to an isomorphism
$C\otimes R\xrightarrow{\cong} C'\otimes R'$
which is also denoted by $\phi$. 
Since $ A, A' \in {\mathcal C}_{3}$, their 
centers, namely $C$ and $C'$, are domains.
By Proposition \ref{xxpro2.3}(1),
$$R=u_{good}(R)\subseteq
u_{good}(C\otimes R)\subseteq u_{good}(C)
\otimes u_{firm}(R)=R,$$ 
which implies that $R=u_{good}(R)=
u_{good}(C\otimes R)$. Similarly,
$R'=u_{good}(R')=u_{good}(C'\otimes R')$.
Now
$$ R'=u_{good}(R') 
=u_{good}(C'\otimes R')
\xrightarrow{\phi^{-1}}
u_{good}(C\otimes R)=R,$$
which implies that $R$ is isomorphic to $R'$. 
Let $I$ be an ideal of $R$ that has 
codimension one and let $J=\phi(I)$. Then 
$$A\cong A\otimes \Bbbk\cong A\otimes (R/I)
=(A\otimes R)/(I)\xrightarrow{\phi/(I)} 
(A'\otimes R')/(J)=A'\otimes (R'/J)
\cong A'\otimes \Bbbk\cong A'$$
as required. 

(2) As in the setting of the proof of part (1), 
an algebra isomorphism $\phi: A\otimes R\to 
A'\otimes R'$ restricts to an algebra 
isomorphism $\phi\mid_{R}: R\to R'$. Now let 
$A=A'$ and $R=R'$. Then every algebra 
automorphism $\phi\in \Aut_{alg}(A\otimes R)$ 
restricts to an algebra automorphism 
$\phi\mid_{R}\in \Aut_{alg}(R)$. Hence the map
$\Psi:\phi\to\phi\mid_{R}$ is a group homomorphism $\Aut_{alg}(A\otimes R)\to \Aut_{alg}(R)$.
For every $\sigma\in \Aut_{alg}(R)$, 
$\phi:=Id\mid_{A}\otimes \sigma$ is an 
automorphism of $A\otimes R$ such that 
$\sigma=\phi\mid_{R}$. So $\Psi$ is surjective. It 
is clear that the kernel of $\Psi$ is 
$\Aut_{alg}(A\otimes R\mid R)$. So the short 
exact sequence follows.
\end{proof}

\section{Stable subspaces and Characterization Problem}
\label{xxsec3}

First we recall a definition. A subspace (resp. 
subalgebra) $V$ of $A$ is called {\sf $\Aut$-stable} 
if for every automorphism $\sigma$ of $A$,
$\sigma(V)\subseteq V$.

\begin{lemma}
\label{xxlem3.1}
For every property $P$, $u_{P}(A)$ is an 
$\Aut$-stable subalgebra of $A$.
\end{lemma}

\begin{proof}
Let $\sigma$ be an algebra isomorphism from $A$ to 
$B$. Suppose there is a filtration ${\mathbf F}$ 
of $B$. Define a filtration, still denoted by 
${\mathbf F}$, of $A$, by
$$F_i(A)=\sigma^{-1} F_i(B)$$
for all $i\geq 0$. Then $\sigma$ maps $F_0(A)$ 
to $F_0(B)$. So $\sigma$ induces an algebra 
isomorphism from $u_{P}(A)\to u_{P}(B)$. When 
$A=B$, $\sigma$ maps $u_{P}(A)$ to itself. This 
means that $u_{P}(A)$ is $\Aut$-stable. 
\end{proof}

In general an algebraic (or geometric) 
invariant of an algebra is always $\Aut$-stable. 

\begin{Ex}
\label{xxexa3.2}
A connected ${\mathbb N}$-graded noncommutative 
algebra has a nontrivial $\Aut$-stable subalgebra. 
To see this let $A$ be such an algebra with graded 
maximal ideal $I:=A_{\geq 1}$. Let $B$ be the 
subalgebra generated by commutators $xy-yx$ for 
all $x, y\in A$. Clearly $B$ is an $\Aut$-stable 
subalgebra. Since $A$ is connected graded, 
$B_{\geq 1}$ is a subspace of $I^2$ which is a 
nontrivial subspace of $I$. This implies that 
$B\neq A$. Since $A$ is not commutative, 
$B\neq \Bbbk$. So $B$ is a proper subalgebra. 
For example, any free algebra $\Bbbk\langle x_1,
\cdots, x_n\rangle$ with $n\geq 2$ has some 
nontrivial $\Aut$-stable subalgebra.
\end{Ex}

Usually it is easy to construct $\Aut$-stable 
subalgebras/subspaces of an algebra. However 
the next result says that it is possible for 
some algebras to have no  nontrivial $\Aut$-stable 
subspace.

\begin{thm}
\label{xxthm3.3}
Suppose ${\rm{char}}\; \Bbbk=0$. Let $A:=
\Bbbk[z_1,\cdots,z_m]$ for some $m\geq 1$
{\rm{(}}$m$ may be infinite{\rm{)}}.
\begin{enumerate}
\item[(1)]
If $m\geq 2$, $A$ has no nontrivial $\Aut$-stable 
subspace.
\item[(2)]
If $m=1$, $A$ has no nontrivial $\Aut$-stable 
subalgebra and every nontrivial  $\Aut$-stable 
subspace is of the form $(\Bbbk+\Bbbk z_1)^d$ for 
some $d\geq 1$.
\end{enumerate}
\end{thm}

\begin{proof} (1) Let $B\neq 0, \Bbbk$ be an 
$\Aut$-stable subspace of $A$. We need to show 
that $B=A$.  

\medskip
\noindent
Claim 1: $B$ contains $V$ where $V:=\Bbbk+
\sum_{i=1}^{m} \Bbbk z_i$. 

\noindent
{\it {Proof of Claim 1:}}
Let $f\in B\setminus \Bbbk$. There is a standard 
filtration/grading of $A$ where $\deg z_i = 1$ 
for all $i$.  If $\deg f=1$, then 
up to an automorphism, we may assume $f=z_1$.
For any fixed $i_0,j_0$, there is an algebra 
automorphism that sends $z_{i_0}$ to $z_{j_0}$.
So $z_i\in B$ for all $i$. Considering the 
automorphism $\sigma: z_1\to z_1+1$ and
$z_i\to z_i$ for all $i>1$, we obtain that
$1=\sigma(z_1)-z_1\in B$. Thus, $B\supseteq V$ 
and Claim 1 holds. Now we use induction on 
$\deg f$ and suppose $\deg f\geq 2$. There
are two cases to consider. Case 1: there is an 
$i$ such that $d:=\deg_{z_i}(f)\geq 2$. In this 
case write $f$ as $\sum_{s=0}^{d} z_i^s h_s$ 
where $h_s$ are polynomials in $z_1,\cdots, 
\widehat{z_i},\cdots, z_m$. Case 2: 
$\deg_{z_j}(f)\leq 1$ for all $j$. Choose $i$ 
so that $\deg_{z_i}(f)=1$ and write $f$ as 
$z_i h_1+h_0$ where $h_0,h_1$ are polynomials 
in $z_1,\cdots, \widehat{z_i},\cdots, z_m$
with $h_1\not\in \Bbbk$. In both cases, let 
$\sigma$ be the isomorphism $\sigma: z_i\to 
z_i+1$ and $z_j\to z_j$ for all $j\neq i$. 
Then $g:= \sigma(f)-f \not\in \Bbbk$ and has 
degree at most $\deg f-1$. Since $B$ is 
$\Aut$-stable, $g\in B$. The assertion follows
by induction. 

\medskip
\noindent
Claim 2: When $m\geq 2$, every monomial is in $B$.

\noindent
{\it {Proof of Claim 2:}}
We
will consider only the case when $m$ is finite. So we need to show that $z_1^{d_1}\cdots z_{m}^{d_m}\in B$ for all $d_i$.
By Claim 1, $z_i\in B$. If $d_m=0$,
consider the automorphism $\tau: z_i\to z_i$
for all $i<m$ and $z_m\to z_m +z_1^{d_1}\cdots 
z_{m-1}^{d_{m-1}}$, one sees that $z_1^{d_1}
\cdots z_{m-1}^{d_{m-1}}=\tau(z_m)-z_m\in B$. 
Using automorphisms, this implies that if we 
have $m-1$ linearly independent elements 
$y_1,\cdots,y_{m-1}$ in $V_0=\sum_{i=1}^m 
\Bbbk z_i$, then 
$y_1^{d_1}\cdots y_{m-1}^{d_{m-1}}\in B$.
Let $c_1,\cdots, c_{n+1}$ be distinct 
scalars in $\Bbbk$. Then the determinant 
of the Vandermonde matrix 
\[\begin{pmatrix}
1 & c_1 & \cdots & c_{1}^{n-1} & c_1^{n}\\
1 & c_2 & \cdots & c_{2}^{n-1} & c_2^{n}\\
\vdots & \vdots & \ddots & \vdots & \vdots\\
1 & c_n & \cdots & c_{n}^{n-1} & c_n^{n}\\
1 & c_{n+1} & \cdots & c_{n+1}^{n-1} & c_{n+1}^{n}
\end{pmatrix}\] 
is nonzero. By linear algebra, 
$z_1^{d_1}z_2^{d_2} \cdots z_m^{d_m}$ is in 
the $\Bbbk$-linear span of elements, for different 
$c_1,\cdots, c_{n+1}\in \Bbbk$, 
$$(z_1+c_s z_2)^{d_1+d_2}z_{3}^{d_3}
\cdots z_{m}^{d_m}
=\sum_{j=0}^{d_1+d_2} c_s^j \binom{d_1+d_2}{j} z_1^{d_1+d_2-j} z_2^j z_3^{d_3}
\cdots z_{m}^{d_m},$$ 
which are in $B$. So Claim 2 holds. 

Part (1) follows from Claim 2.

(2) When $A=\Bbbk[z_1]$ we claim that if $f\in 
B\setminus \Bbbk$ has degree $d$, then $B$ 
contains $(\Bbbk+\Bbbk z_1)^{d}$. The proof is similar to the proof of part (1) where we did not assume $m\geq 2$, and the details are omitted.
\end{proof}

Based on the previous result, we ask the following question: 

\begin{ques}
\label{xxque3.4}
Are there other affine commutative algebras that 
have no nontrivial  $\Aut$-stable subspace?
\end{ques}

 The center of an algebra is clearly an $\Aut$-stable subalgebra.
Besides this, Theorem~\ref{xxexa3.2} shows that a noncommutative 
connected graded algebra does admit a nontrivial $\Aut$-stable subspace. 
However, it is also possible to find a noncommutative affine algebra 
with no nontrivial $\Aut$-stable subspace; see the next theorem.

\begin{thm}
\label{xxthm3.5}
Suppose ${\rm{char}}\; \Bbbk=0$.
\begin{enumerate}
\item[(1)]
The $n$th Weyl algebra has no nontrivial 
$\Aut$-stable subspace where $n\geq 1$.
\item[(2)]
Let $C$ be the tensor product of 
$A:=\Bbbk[z_1,\cdots,z_m]$
for $m\geq 2$ and the $n$th Weyl algebra. 
Then the only nontrivial  $\Aut$-stable subspace 
of $C$ is $A\otimes \Bbbk$,  which is the center 
of $C$.
\end{enumerate}
\end{thm}

\begin{proof} 
(1) Recall that the $n$th Weyl algebra is 
${\mathbb A}_n:=\Bbbk\langle x_1,\cdots, x_n, y_1, 
\cdots, y_n\rangle/(I)$ where $I$ is the ideal 
generated by relations $x_i x_j-x_j x_i=0$, 
$y_i y_j-y_jy_i=0$ and $x_i y_j-y_j x_i=\delta_{ij}$ 
for all $i,j$. There are some obvious automorphisms (their inverse maps are easy to see): 
\begin{enumerate}
\item[(aut1)]
$\tau : x_1 \to  2x_1, y_1 \to  \frac{1}{2} y_1$ and $x_j \to x_j$, $y_j  \to  y_j$ for all $j > 1$. 
\item[(aut2)]
For any permutation $\sigma\in S_{n}$, $f_{\sigma}: 
x_i\to x_{\sigma(i)}, y_i\to y_{\sigma(i)}$.
\item[(aut3)]
Let $M:=(m_{ij})$ be an invertible $n\times n$ 
matrix. Let $N:=(n_{ij})$ be the matrix $(M^{-1})^{\tau}$. Then 
$\phi_{M}: x_i\to \sum_{j=1}^n m_{ij} x_j, y_{k}\to 
\sum_{l=1}^n n_{kl} y_l$ is an automorphism 
of $B$. Both (aut1) and (aut2) are special cases 
of (aut3)
\item[(aut4)]
$\tau_{i_0}: x_{i_0}\to y_{i_0}, y_{i_0}\to -x_{i_0}$ 
and $x_j\to x_j, y_j \to y_j$ for all $j\neq i_0$. 
\item[(aut5)]
Let $h(x_1)$ be a polynomial in $x_1$. Define 
$\alpha_{h}: x_i\to x_i$ for all $i$, 
$y_1\to y_1+h(x_1), y_j\to y_j$ for all $j\neq 1$.
\item[(aut6)]
Let $0\neq c\in \Bbbk$ and $1\leq i\leq n$. Define
$\beta_{i,c}: x_i\to x_i+cy_i, y_{i}\to y_i$ and 
$x_j\to x_j, y_j\to y_j$ for all $j\neq i$.
\end{enumerate}

Let $V_x=\sum_{i=1}^n \Bbbk x_i$, $V_y=\sum_{i=1}^n 
\Bbbk y_i$, and $V=\Bbbk+ V_x+V_y$. Let $B$ be a 
nontrivial $\Aut$-stable subspace of 
${\mathbb A}_n$.

\medskip
\noindent
Claim 1: $V\subseteq B$. 

\noindent{\it {Proof of Claim 1:}}
Since $B$ is nontrivial, there is an element $f\in 
B\setminus \Bbbk$. Define $\deg (x_i)=\deg (y_i)=1$ 
for all $i$. Then $\deg f\geq 1$. If $\deg f=1$, 
write $f=\sum_{i=1}^n(a_i x_i+ b_i y_i) +c$. Without 
loss of generality, we assume that $a_1\neq 0$. 
Then $f_1:=\tau(f)-f= a_1 x_1+(-\frac{1}{2})b_1 y_1
\in B,$ where $\tau$ is the automorphism given in (aut1). 
Similarly, $f_2:=2\tau(f_1)-f_1=3a_1 x_1
\in B$. Therefore $x_1\in B$. Applying automorphism 
$f_{\sigma}$ in (aut2) for different $\sigma$, we 
obtain that $x_i\in B$ for all $i$. Applying 
$\tau _{i_0}$ in (aut4) we have $y_i\in B$ for all $i$. Thus 
$V_x, V_y\subseteq B$. Applying $(\alpha_{h}-Id)$ to 
$V_y$ for $h=1$ where $\alpha_{h}$ is given in (aut5), 
we obtain that $1\in B$. So $V\subseteq B$. Now we use induction on $\deg f$ to conclude that $V \subseteq B$. 
Assume that $\deg f > 1$. Then there are two possibilities: either, 
for some $t := x_i$ or $t := y_i$, we have 
$d := \deg_{t} f > 1$, or, for every  $t = x_i$ or $t = y_i$, 
$d := \deg_{t} f = 1$. 

In the first case, without loss of generality, we may suppose that  
$d = \deg_{y_1} f > 1$, and we write 
\[
   f = \sum_{s=0}^d (y_1)^s h_s,
\]
where $\deg_{y_1}(h_s) = 0$ for each $s$.  

In the second case   without loss of generality, we can find $h_1, h_0 $ such that 
$\deg_{y_1}(h_1) = 0 = \deg_{y_1}(h_0)$,   $h_1 \notin \kk$ and 
\[
   f = y_1 h_1 + h_0. 
\]
In either case, with  $h(x_1) = 1$ in (aut5), we obtain $\alpha_1$ 
and it is not difficult to see that $\alpha_1(f) - f \notin \kk$ has degree less than 
$\deg f$. Hence, we can apply our induction hypothesis to conclude that $V \subseteq B$.

\medskip
\noindent
Claim 2: For every $w\in V_x$ and every $d\geq 1$, 
$w^d\in B$. 

\noindent{\it {Proof of Claim 2:}}
Using automorphism $\phi_{M}$ (given in (aut3)) one 
may assume that $w=x_1$. The claim follows by 
applying $(\alpha_{h}-Id)$ to $V_y$ for $h=x_1^d$ 
(where $\alpha_{h}$ is given in (aut5)).

\medskip
\noindent
Claim 3: $x_1^{d_1}\cdots x_{n}^{d_n}\in B$ for all 
$d_i\geq 0$.

\noindent
{\it {Proof of Claim 3:}}
We prove that $x_1^{d_1}\cdots x_{i}^{d_i}\in B$ by 
induction on $i$. If $i=1$, this is a special
case of Claim 2. Now suppose that the claim holds 
for some $i<n$, namely, $x_1^{d_1}\cdots x_{i}^{d_i}
\in B$ for all $d_1,\cdots, d_i$. Consider the 
automorphism $\phi_{M}$ in (aut3) with $\phi_{M}: 
x_j\to x_j$ for all $j<i$ and $x_i\to x_i+c x_{i+1}$. 
Applying $\phi_{M}$ to $x_1^{d_1}\cdots 
x_{i}^{d_i+d_{i+1}}$ we obtain that
$x_1^{d_1}\cdots x_{i-1}^{d_{i-1}}
(x_{i}+c x_{i+1})^{d_i+d_{i+1}}\in B$. By the 
Vandermonde matrix trick (which was used in the 
proof of Theorem \ref{xxthm3.3}(1)), 
$x_1^{d_1}\cdots x_{i}^{d_{i}} x_{i+1}^{d_{i+1}}\in B$ 
as required.

\medskip
\noindent
Claim 4: $x_1^{d_1} y_1^{t_1}\cdots x_{n}^{d_n} 
y_{n}^{t_n}\in B$ for all $d_i, t_i\geq 0$.

\noindent{\it {Proof of Claim 4:}}
We prove that $(x_1^{d_1}y_{1}^{t_1}\cdots 
x_{i-1}^{d_{i-1}}y_{i-1}^{t_{i-1}}) x_{i}^{d_i} 
\cdots x_{n}^{d_n}\in B$ by induction on $i$. If 
$i=1$, this is Claim 3. Now suppose that the claim 
holds for some $i\leq n$. After changing 
$d_i$ to $d_i+t_i$, we can assume that  
$(x_1^{d_1}y_{1}^{t_1}\cdots 
x_{i-1}^{d_{i-1}}y_{i-1}^{t_{i-1}}) 
x_{i}^{d_i+t_i} x_{i+1}^{d_{i+1}} 
\cdots x_{n}^{d_n}\in B$
for all $d_1,\cdots, d_n$ and all 
$t_1,\cdots, t_{i}$. Applying $\beta_{i,c}$ 
(see (aut6)) to the above element we have that 
\begin{equation}
\label{E3.4.1}\tag{E3.4.1}
(x_1^{d_1}y_{1}^{t_1}\cdots 
x_{i-1}^{d_{i-1}}y_{i-1}^{t_{i-1}}) 
(x_{i}+c y_{i})^{d_i+t_i} x_{i+1}^{d_{i+1}} 
\cdots x_{n}^{d_n}
\in B.
\end{equation}
We will show that $x_i^{d_i} y_i^{t_i}$ 
is in $W$, where $W$ is the $\Bbbk$-linear span 
of all elements $(x_i+c y_i)^{d'}$ for all 
different $c\in \Bbbk$ and $0\leq d'\leq d: 
=d_i+t_i$. We use induction on $d$. 
If $d = 0$, this is a special case of 
Claim 1. Next let $d > 0$. Then, 
using the relation $y_ix_i=x_i y_i-1$, we have 
$$
(x_i+c y_i)^d=
\sum_{s=0}^{d} c^{d-s}\binom{d}{s}  x_i^{s} y_i^{d-s}
+{\text{lower degree terms.}} 
$$
By induction, lower degree terms are in 
$W$. By the Vandermonde matrix trick (see the 
proof of Theorem \ref{xxthm3.3}(1)),
$x_i^{s} y_{i}^{d-s}\in W$ for all $s$. Replacing 
the middle part $(x_{i}+c y_{i})^{d_i+t_i}$ of 
\eqref{E3.4.1} by $x_i^{s} y_i^{d-s}$ we finish 
the original induction. Claim 4 follows by 
setting $i=n+1$.  

Part (1) follows from Claim 4.

(2) Since the center of the algebra $C(:=A\otimes {\mathbb A}_n)$ 
is $A\otimes \Bbbk$ and since the center of any algebra 
is an $\Aut$-stable subspace, $A\otimes \Bbbk$ is 
an $\Aut$-stable subspace of $C$. Let $B\subseteq C$ 
be a nontrivial $\Aut$-stable subspace of $C$. 

\medskip
\noindent
Claim 1: If $B\neq A\otimes \Bbbk$, then $x_1\in B$.

\noindent{\it Proof of Claim 1:}
By the notation introduced in the proof of Part~(1), let  $
   {\mathbb A}_n := \Bbbk\langle x_1, \ldots, x_n, 
   y_1, \ldots, y_n \rangle / (I),$ and $ 
   A := \Bbbk[z_1, \ldots, z_m]$, for some $ m \geq 2$. 
 We 
define $\deg x_i=\deg y_i=\deg z_j=1$ for all 
$i,j$, 
so that $C$ is a filtered algebra.

By the argument in the proof of Theorem 
\ref{xxthm3.3} and the proof of part (1), 
there is an element $f\in B$ of degree 1. 
By identifying $A\otimes \Bbbk$ and $ A$ and similarly 
$ \kk \otimes {\mathbb A}_n $ with ${\mathbb A}_n$, we 
can write it as either $f=z_1 +  x_1$ or $f=z_1$ up 
to an automorphism. In the first case, consider 
an automorphism of $C$, $\phi: z_j\to z_j,
x_i\to x_i+\delta_{i1} z_1 , y_i\to y_i$ for all $i,j$.
Then $z_1=\phi(f)-f\in B$. In both cases we can 
assume that $z_1\in B$. 

Since $A$ has no nontrivial $\Aut$-stable subspace 
[Theorem \ref{xxthm3.3}(2)], $B$ contains 
$A\otimes \Bbbk$. By hypothesis, $V$ strictly 
contains $A\otimes \Bbbk$ as a subspace. By an 
argument that is similar to the one given in the 
proof of Theorems \ref{xxthm3.3} and the proof 
of part (1), there is an element $f\in B$ of 
degree 1 that is not in $A\otimes \Bbbk$. Then 
we can assume $f=x_1+z_1$. Up to an automorphism, 
one may further assume that $f=x_1\in B$. 

\medskip
\noindent
Claim 2: $B=A\otimes \Bbbk$.

\noindent
{\it Proof of Claim 2:} Suppose to the contrary that
$B\neq A\otimes \Bbbk$. By Claim 1 and its proof,
$B\supseteq A\otimes \Bbbk$ and $x_1\in B$.

Since ${\mathbb A}_n$ has no nontrivial $\Aut$-stable 
subspaces, $\Bbbk \otimes {\mathbb A}_n\subseteq B$. 
Consider the automorphism
$\tau: z_j\to z_j, x_i\to x_i, y_i\to y_i+\delta_{1i} 
z_1 \otimes x_1$ for all $j,i$. One sees that 
$z_1\otimes x_1=\tau(x_1)-x_1\in B$. Let 
$S:=\{a\in A\mid a\otimes x_1\in B\}$. Since $B$ 
is an $\Aut$-stable subspace of $A\otimes {\mathbb A}_n$, 
$S$ is an $\Aut$-stable subspace of $A$ which 
contains $z_1$. By Theorem \ref{xxthm3.3}(2), $S=A$, 
or $A\otimes x_1\subseteq B$. For every $x\in A$, 
let $T_{x}:=\{b\in {\mathbb A}_n\mid 
x\otimes b\in B\}$. Since $x_1\in T_{x}$, 
$T_{x}={\mathbb A}_n$ by part (1). Thus 
$B=A\otimes {\mathbb A}_n$, yielding a 
contradiction. Therefore $B=A\otimes \Bbbk$.
\end{proof}

We wish to use $\Aut$-stable subspaces to 
characterize polynomial rings, aiming to 
answer Question \ref{xxque0.10}. Here is a  
version of such a characterization.

\begin{thm}
\label{xxthm3.6}
Suppose $\Bbbk$ is an algebraically closed 
field of characteristic zero and $A$ is an 
algebra. Then $A$ is isomorphic to 
$\Bbbk[z_1,\cdots,z_m]$ for some integer $m\geq 2$ 
if and only if the following two conditions hold:
\begin{enumerate}
\item[(1)]
$A\neq \Bbbk$ is affine and connected graded,
\item[(2)]
$A$ has no nontrivial $\Aut$-stable subspace.
\end{enumerate}
\end{thm}

\begin{proof}
If $A$ is the polynomial ring $\Bbbk[z_1,\cdots,z_m]$ 
for some $m\geq 2$, then it is finitely generated and 
connected graded. By Theorem \ref{xxthm3.3}, it has no 
nontrivial $\Aut$-stable subspace. So (1) and (2) hold.

Conversely, assume that $A$ is an algebra satisfying 
(1) and (2). We need to show that 
$A\cong \Bbbk[z_1,\cdots,z_m]$ for some integer $m\geq 2$. 
By Example \ref{xxexa3.2}, $A$ is commutative.

Let $A\neq \Bbbk$ be an affine connected graded
commutative algebra. Let $R$ be the prime radical 
(also equal to the nilradical in this case) of $A$. 
Since $R$ is $\Aut$-stable, it must be zero, 
or equivalently, $A$ is semiprime. Let $
\{{\mathfrak p}_1, \cdots, {\mathfrak p}_n\}$
be the set of all minimal primes. Since $A$ is 
graded, so are ${\mathfrak p}_i$ for all $i$.
Thus $V:=\sum_{i=1}^{n} {\mathfrak p}_i$ is an
$\Aut$-stable subspace that is a subspace of
$A_{\geq 1}$. Hence $V=0$ which implies that $A$ 
is a domain. We claim that $A$ is regular 
(namely, $A$ has finite global dimension). Suppose to 
the contrary that $A$ is not regular. Let $S$ be 
the non-empty singular locus of $\Spec A$ which 
is a closed subscheme of $\Spec A$.
Since $A$ is 
an affine domain, there is a nonzero ideal 
$I\subsetneq A$ such that $S =\Spec A/I$. So 
$I$ is an $\Aut$-stable subspace of $A$. Hence 
$I=A$, yielding a contradiction. Therefore $A$ 
is regular. Since every regular connected graded 
commutative algebra is isomorphic to 
$\Bbbk[z_1,\cdots, z_m]$ for some $m\geq 1$ 
\cite[Theorem 1.4]{Ro-2024}, the assertion follows 
by verifying $m\neq 1$ [Theorem \ref{xxthm3.3}(2)].
\end{proof}

\begin{remark}
\label{xxrem3.7} 
We have the following comments about the above 
theorem.
\begin{enumerate}
\item[(1)]
$\Bbbk[x_1]$ has no  nontrivial $\Aut$-stable 
subalgebra, but has nontrivial $\Aut$-stable 
subspaces, see Theorem {\rm{\ref{xxthm3.3}(2)}}. 
\item[(2)]
If ${\text{char}}\; \Bbbk=p>0$, then 
$\Bbbk[z_1,\cdots,z_m]$ has $\Aut$-stable 
subalgebras $\Bbbk[z_1^{p^d},\cdots,z_{m}^{p^{d}}]$
for each $d\geq 1$.
\item[(3)]
When $A$ is noncommutative, condition {\rm{(1)}} 
in Theorem {\rm{\ref{xxthm0.12}}} does not 
follow from condition {\rm{(2)}}, see Theorem 
{\rm{\ref{xxthm3.5}(1)}}.
\item[(4)]
When $A$ is an affine commutative 
ring, we conjecture that condition {\rm{(1)}} 
in Theorem {\rm{\ref{xxthm0.12}}} follows 
from condition {\rm{(2)}}. If this is the case, 
then a commutative algebra $A\neq \Bbbk$ is 
$\Bbbk[z_1,\cdots,z_m]$ for some $m\geq 2$ 
if and only if $A$ does not have any nontrivial 
$\Aut$-stable subspaces.
\end{enumerate}
\end{remark}

\subsection*{Acknowledgments.} 
Nazemian was supported by the Austrian Science Fund (FWF), grant P 36742. She would also like to thank the organizers of the “Poisson Geometry and Artin–Schelter Regular Algebras Workshop,” Hangzhou, for the invitation where this joint work was initiated. Wang was supported by the NSF of China (No. 11971289).
Zhang was partially supported by the US National Science Foundation
(No. DMS-2302087).  The authors thank the two anonymous referees for their careful reading of the manuscript.

\bibliography{main}

\providecommand{\bysame}{\leavevmode\hbox to3em{\hrulefill}\thinspace}
\providecommand{\MR}{\relax\ifhmode\unskip\space\fi MR }
\providecommand{\MRhref}[2]{%
  \href{http://www.ams.org/mathscinet-getitem?mr=#1}{#2}
}
\providecommand{\href}[2]{#2}
\begin{thebibliography}{10}

\bibitem{AEH-1972}
S.~S. Abhyankar, W.~Heinzer, and P.~Eakin, \emph{On the uniqueness of the
  coefficient ring in a polynomial ring}, J. Algebra \textbf{23} (1972),
  310--342. \MR{306173}

\bibitem{AS-1995}
M.~Artin and J.~T. Stafford, \emph{Noncommutative graded domains with quadratic
  growth}, Invent. Math. \textbf{122} (1995), no.~2, 231--276. \MR{1358976}

\bibitem{As-1987}
T.~Asanuma, \emph{Polynomial fibre rings of algebras over {N}oetherian rings},
  Invent. Math. \textbf{87} (1987), no.~1, 101--127. \MR{862714}

\bibitem{BZ1-2017}
J.~Bell and J.~J. Zhang, \emph{Zariski cancellation problem for noncommutative
  algebras}, Selecta Math. (N.S.) \textbf{23} (2017), no.~3, 1709--1737.
  \MR{3663593}

\bibitem{Fu-1979}
T.~Fujita, \emph{On {Z}ariski problem}, Proc. Japan Acad. Ser. A Math. Sci.
  \textbf{55} (1979), no.~3, 106--110. \MR{531454}

\bibitem{GWY2022}
J.~Gaddis, X.~Wang, and D.~Yee, \emph{Cancellation and skew cancellation for
  {P}oisson algebras}, Math. Z. \textbf{301} (2022), no.~4, 3503--3523.
  \MR{4449718}

\bibitem{Gu2}
N.~Gupta, \emph{On the cancellation problem for the affine space {$\Bbb{A}^3$}
  in characteristic {$p$}}, Invent. Math. \textbf{195} (2014), no.~1, 279--288.
  \MR{3148104}

\bibitem{Gu1}
\bysame, \emph{On {Z}ariski's cancellation problem in positive characteristic},
  Adv. Math. \textbf{264} (2014), 296--307. \MR{3250286}

\bibitem{Gu3-2015}
\bysame, \emph{A survey on {Z}ariski cancellation problem}, Indian J. Pure
  Appl. Math. \textbf{46} (2015), no.~6, 865--877. \MR{3438041}

\bibitem{HTW-2024}
H.~Huang, X.~Tang, and X.~Wang, \emph{A survey on {Z}ariski cancellation
  problems for noncommutative and {P}oisson algebras}, Recent advances in
  noncommutative algebra and geometry, Contemp. Math., vol. 801, Amer. Math.
  Soc., [Providence], RI, [2024] \copyright 2024, pp.~125--141. \MR{4756381}

\bibitem{Kr-1996}
H.~Kraft, \emph{Challenging problems on affine {$n$}-space}, S\'eminaire
  Bourbaki, Vol.\ 1994/95, no. 237, Soci\'et\'e Math\'ematique de France, 1996,
  pp.~Exp. No. 802, 5, 295--317. \MR{1423629}

\bibitem{KL-2000}
G.~Krause and T.~H. Lenagan, \emph{Growth of algebras and {G}elfand-{K}irillov
  dimension}, revised ed., Graduate Studies in Mathematics, vol.~22, American
  Mathematical Society, Providence, RI, 2000. \MR{1721834}

\bibitem{LWZ2019}
O.~Lezama, Y.~Wang, and J.~J. Zhang, \emph{Zariski cancellation problem for
  non-domain noncommutative algebras}, Math. Z. \textbf{292} (2019), no.~3-4,
  1269--1290. \MR{3980292}

\bibitem{MR-2001}
J.~C. McConnell and J.~C. Robson, \emph{Noncommutative {N}oetherian rings},
  revised ed., Graduate Studies in Mathematics, vol.~30, American Mathematical
  Society, Providence, RI, 2001, With the cooperation of L. W. Small.
  \MR{1811901}

\bibitem{MS-1980}
M.~Miyanishi and T.~Sugie, \emph{Affine surfaces containing cylinderlike open
  sets}, J. Math. Kyoto Univ. \textbf{20} (1980), no.~1, 11--42. \MR{564667}

\bibitem{Ro-2024}
D.~Rogalski, \emph{Artin-{S}chelter regular algebras}, Recent advances in
  noncommutative algebra and geometry, Contemp. Math., vol. 801, Amer. Math.
  Soc., [Providence], RI, [2024] \copyright 2024, pp.~195--241. \MR{4756385}

\bibitem{RS-2013}
L.~Rowen and D.~J. Saltman, \emph{Tensor products of division algebras and
  fields}, J. Algebra \textbf{394} (2013), 296--309. \MR{3092723}

\bibitem{Ru-1981}
P.~Russell, \emph{On affine-ruled rational surfaces}, Math. Ann. \textbf{255}
  (1981), no.~3, 287--302. \MR{615851}

\bibitem{TV-1996}
J.~Tate and M.~Van~den Bergh, \emph{Homological properties of {S}klyanin
  algebras}, Invent. Math. \textbf{124} (1996), no.~1-3, 619--647. \MR{1369430}

\end{thebibliography}
\bibliographystyle{amsplain}

\end{document}